\documentclass[a4paper, 12pt]{amsart}
\usepackage{amsmath, amsthm, amscd, amssymb, amsfonts, amsxtra, amssymb, latexsym}
\usepackage{enumerate}
\usepackage{verbatim}

\usepackage{array}
\usepackage{hyperref}
\hypersetup{colorlinks = true,	allcolors  = blue}

\newcommand{\G}{\Gamma}
\newcommand{\Z}{\mathbb{Z}}

\newcommand{\N}{\mathbb{N}}
\newcommand{\ff}{\mathbb{F}}

\newcommand{\sk}{\smallskip}

\newtheorem{thm}{Theorem}[section]
\newtheorem{prop}[thm]{Proposition}

\newtheorem{coro}[thm]{Corollary}

\theoremstyle{definition}
\newtheorem{rem}[thm]{Remark}
\newtheorem{exam}[thm]{Example}

\theoremstyle{remark}

\usepackage{color}

\begin{document} \sloppy
\numberwithin{equation}{section}
\title[Numb.\@ of cliques of Paley-type graphs over finite commutative local rings]{Number of cliques of Paley-type graphs \\ over finite commutative local rings} 
\author[A.L.\@ Gallo, D.E.\@ Videla]{Andrea L.\@ Gallo, Denis E.\@ Videla}
\dedicatory{\today}
\thanks{Partially supported by CONICET, FONCyT and SECyT-UNC}

\address{Andrea L.\@ Gallo, FaMAF -- CIEM (CONICET), Universidad Nacional de C\'ordoba, \newline
	Av.\@ Medina Allende 2144, Ciudad Universitaria, (5000) C\'ordoba, Argentina. 
	\newline {\it E-mail: andregallo88@gmail.com}}
\address{Denis E.\@ Videla, FaMAF -- CIEM (CONICET), Universidad Nacional de C\'ordoba, \newline
	Av.\@ Medina Allende 2144, Ciudad Universitaria,  (5000) C\'ordoba, Argentina. 
	\newline {\it E-mail: devidela@famaf.unc.edu.ar}}

\begin{abstract}
In this work, given $(R,\frak m)$ a finite commutative local ring with identity and $k \in \N$ with $(k,|R|)=1$, 
we study the number of cliques of any size in the Cayley graph 
$G_R(k)=Cay(R,U_R(k))$ 
with $U_R(k)=\{x^k : x\in R^*\}$. 
Using the known fact that the graph $G_R(k)$ can be obtained by blowing-up the vertices of $G_{\ff_{q}}(k)$ a number $|\frak{m}|$ of times, 
with independence sets the cosets of $\frak{m}$, where $q$ is the size of the residue field $R/\frak m$ (see \cite{PV9}).
Then, by using the above blowing-up, 
we reduce the study of the number of cliques in $G_R(k)$ over the local ring $R$ to the computation of the number of cliques of $G_{R/\frak{m}}(k)$ over the finite residue field $R/\frak m \simeq \ff_q$.
In this way, using known numbers of cliques of generalized Paley graphs ($k=2,3,4$ and $\ell=3,4$), 
we obtain several explicit results for the number of cliques 
over finite commutative local rings with identity.
\end{abstract}

\maketitle

\section{Introduction}
In this work we study the number of cliques of  $G_R(k)$ over finite commutative local rings with identity $R$, where $(k,|R|)=1$, 
through the decomposition of $G_R(k)$ in terms of certain graph operation called blow-up. 
The case when $k$ and $|R|$ are not coprime remains open. 
We will reduce the computation of the number of $\ell$-cliques $\mathcal{K}_{\ell}(G_R(k))$ over finite commutative local rings $(R,\frak m)$ 
to the problem of computing  $\mathcal{K}_{\ell}(\G(k,q))$ where $\G(k,q)$ is the generalized Paley graph. 
More precisely, we will show that $\mathcal{K}_{\ell}(G_R(k))$ can be put in terms of $\mathcal{K}_{\ell}(\G(k,q))$, 
where the residue field $R/\frak m$ is isomorphic to $\ff_q$.

\subsection*{Preliminaries}

Let $G$ be a finite abelian group and $S$ a subset of $G$ with $0\notin S$. The \textit{Cayley graph } $\G=Cay(G,S)$ is the directed graph whose vertex set is $G$ and $v, w \in G$ form a directed edge (or arc) $\overrightarrow{vw}$ of $\Gamma$ from $v$ to $w$ if $w-v \in S$. Since $0\notin S$ then $\Gamma$ has no loops. 
Notice that if $S$ is symmetric, that is $-S=S$, then we can consider $Cay(G,S)$ as undirected (and conversely), and hence $Cay(G,S)$ is $|S|$-regular. 
In this work, we always consider $S$ symmetric.


We are interested in the \textit{unitary Cayley graph of $k$-th powers} or \textit{$k$-unitary Cayley graph} defined by 
\begin{equation} \label{GRk}
	G_R(k)= Cay(R,U_R(k)) \qquad \text{where} \qquad U_R(k) = \{x^k: x\in R^*\}.
\end{equation} 
The case $G_R(1)$ are the called  unitary Cayley graphs, they were extensively studied  (see \cite{Ak+}, \cite{Il}, \cite{Ki+}, \cite{LZ} and \cite{PV5}). 
Recently, Liu and Zhou \cite{LZ2} defined and studied the \textit{quadratic unitary Cayley graphs} 
$\mathcal{G}_R = Cay(R,T_{R})$,
where $T_R=Q_{R}\cup (-Q_R)$ and $Q_{R}=\{x^2: x\in R^*\}$ with $R$ a finite commutative ring with identity, when $Q_R$ is symmetric then $\mathcal{G}_R=G_{R}(2)$. 
Some structure properties of $G_{R}(k)$, in the case $R$ local was strudied in \cite{PV9} (see also \cite{BB} for $R=\mathbb{Z}_{p^{\alpha}}$).
In general, notice that $G_{R}(k)$ is directed, moreover $U_R(k)$ is a symmetric set if and only if $-1\in U_R(k)$. 


One other interesting instance of these graphs is when $R=\ff_q$ is a finite field of cardinality $q=p^r$ with $p$ prime, 
where $k$ a non-negative integer (with $k\mid q-1$), 
this graph is called \textit{generalized Paley graph} (\textit{GP-graph} for short) and is denoted by $\G(k,q)$,
more precisely 
\begin{equation} \label{Gammas}
\G(k,q) = Cay(\ff_{q},U_{k}) \qquad \text{with} \qquad U_k = \{ x^{k} : x \in \ff_{q}^*\}.
\end{equation} 
Notice that $\G(k,q)$ is an $n$-regular graph with $n=\tfrac{q-1}k$. 
The graph $\G(k,q)$ is undirected either if $q$ is even or if $k$ divides $\tfrac{q-1}2$ when $p$ is odd (equivalently if $n$ is even when $p$ is odd) 
and it is connected if $n$ is a primitive divisor of $q-1$ (see  \eqref{prim div}).  
When $k=1$ we get the complete graph $\G(1,q)=K_q$ and when $k=2$ we get the classic Paley graph $\Gamma(2,q) = P(q)$. 
The graphs $\G(3,q)$ and $\G(4,q)$ are also of interest (see \cite{PV8}).
The GP-graphs have been extensively studied in the few past years (see for instance \cite{LP}, \cite{PP}, \cite{PV6}, \cite{PV7}, \cite{PV4}. \cite{V}, \cite{Y1}, \cite{Y2}.) 

\subsubsection*{Number of cliques in generalized Paley graphs}

Recall that a complete graph of $n$ vertices is an undirected graph such that all of its vertices are neighbours, 
it is denoted by $K_{n}$.

Given $G$ be a graph, a clique in $G$ is a complete subgraph of $G$. The clique number of $G$ is the size of its maximum clique in $G$ and is usually denoted by $\omega(G)$.
The number of cliques of size $\ell$ in $G$ is denoted by $\mathcal{K}_{\ell}(G)$, notice that 
\begin{equation}\label{clique numb}
	\omega(G)=\max \{\ell \in \mathbb{N}: \mathcal{K}_{\ell}(G) \neq 0\}.
\end{equation}	

The clique number of $\G(k,q)$ was studied in  
\cite{Y1}, \cite{Y2}. 
In \cite{EPS}, the authors found a closed formula for $\mathcal{K}_{\ell}(\G(2,q))$ for $\ell=3,4$ and $q=p$ prime, 
in terms of certain binary quadratic forms, see also \cite{ABLMP} for an extension to $q$ a prime power.
In \cite{DM}, the authors found a general formula for $\mathcal{K}_{\ell}(\G(k,q))$ for $\ell=3,4$ in terms of hypergeometric sums, they also
gave more explicit formulas for $k = 2,3,4$. 

In the case of $G_{R}(k)$ when $R$ is not a finite field, 
recently Bhowmik and Barman studied the number of cliques for the local ring $R=\Z_{p^{\alpha}}$ with $k=2$ (see \cite{BB}).
For $R$ non-local, Das studied the number of cliques for $R=\Z_{pq}$ with $k=2$ and $p,q$ primes, this formula was recently generalized for $R=\Z_{n}$ for $n$ odd in \cite{BB2}.
 
 
\subsection*{Outline and main results}
In what follows let $R$ be a finite commutative ring with identity and let $k\in \N$ be coprime with $|R|$. We now give a brief account of the main results in the paper. 
 
In Section \ref{sec2}, we recall some facts about $t$-balanced blow-up operations 
and show a general reduction formula for $\mathcal{K}_{\ell}(G^{(m)})$ in terms of $\mathcal{K}_{\ell}(G)$. Thus, we recall the structure of the graph $G_R(k)$ for $R$ a local ring with unique maximal ideal $\frak m$. In this case, we have that $G_R(k)=G_{R/\frak{m}}(k)^{(m)}$ where $m=|\frak{m}|$ and so,
by using the formula obtained previously, 
we show that we can compute $\mathcal{K}_{\ell}(G_{R}(k))$ in terms of 
$\mathcal{K}_{\ell}(G_{R/\frak{m}}(k))$ and $m$. 

In Section \ref{sec 4}, we study the number of complete subgraphs of size $3$ (triangles) of $G_{R}(k)$ for $k=2,3,4$. 
We also show that in the case $R$ a finite field, 
we obtain a derived property in field extensions 
for the number of cliques, 
more precisely for $k=3$ and $4$ we show that we can derive the values of $\mathcal{K}_{3}(\G(k,q^{\ell}))$ 
from the value of $\mathcal{K}_{3}(\G(k,q))$, 
under some hypothesis. 

Finally, in section \ref{sec 5} we find a close formula for the number of clique of size $4$
for $G_{R}(2)$. In the same way as in $3$-cliques, we obtain a derived property in field extensions 
for the number of $4$-cliques.

\section{Balanced blow-ups and cliques in $G_{R}(k)$}\label{sec2}	
	
\subsubsection*{Balanced blow-ups and cliques}	
For any graph $G$ and $m\in \N$, the \textit{(balanced) blow-up of order $m$} of $G$, denoted  $G^{(m)}$, is the graph obtained by replacing each vertex $x$ of $G$ by a set $V_x$ of $m$ independent vertices and every edge $\{x, y\}$ of $G$ by a complete bipartite graph $K_{m,m}$ with parts $V_x$ and $V_y$ (of course $G^{(1)} = G$). 
Notice that we have the natural isomorphism 
\begin{equation} \label{blowup}
	G \otimes \mathring{K}_m \simeq G^{(m)},
\end{equation}
where $\mathring{\G}$ denotes a graph $\G$ with a loop added at each vertex. 
 
\begin{prop}\label{blow clique}
	Let $G$ be a simple undirected graph and let $m,\ell\in \N$. Then, 
	$$\mathcal{K}_{\ell}(G^{(m)})=\mathcal{K}_{\ell}(G) \cdot m^{\ell}.$$
\end{prop}
\begin{proof}
	Notice that any clique of size $\ell $ of $G^{(m)}$ induces a clique of the same size in $G$. Indeed the vertices of this clique in $G^{(m)}$ must belong to different 
	independent sets of the blow-up, since if two of them belong to the same independent set, so they are not neighbors.
	Conversely, given a clique $K_{\ell}$ in $G$, by blowing up of its vertices we can obtain 
	$|m|^{\ell}$ distinct cliques in $G^{(m)}$. 
	The above argument shows that all of the cliques in $G^{(m)}$ can be obtained in this way.
	Hence, by the so called multiplication principle of combinatorics we have that
		$$\mathcal{K}_{\ell}(G^{(m)})=\mathcal{K}_{\ell}(G) \cdot m^{\ell},$$
	as asserted.
\end{proof}

\begin{exam}
For $\ell,m,n\in \N$ we have that 
$$\mathcal{K}_{\ell}(K_{n}^{(m)})= \tbinom{n}{\ell}m^{\ell}.$$
Indeed, by taking into account that any choice of $\ell$ vertices in $K_n$ determines a clique in $K_n$ 
and all of the cliques in $K_n$ are determined uniquely in this way, 
we have that $\mathcal{K}_{\ell}(K_n)=\binom{n}{\ell}$. 
The assertion follows immediately from the above proposition.
\end{exam}

As a consequence of the previous proposition we obtain the following.
\begin{coro}\label{coro clique}
		Let $G$ be a simple undirected graph. 
		Then, $\omega(G^{(m)})=\omega(G)$ for any $m\in \N$.
\end{coro}
\begin{proof}
	By Proposition \ref{blow clique} we have that $\mathcal{K}_{\ell}(G^{(m)})=0$ if and only if $\mathcal{K}_{\ell}(G)=0$.
	The corollary follows directly from \eqref{clique numb}.
\end{proof}

\subsubsection*{The structure of $G_R(k)$ and the number $\mathcal{K}_{\ell}(G_R(k))$ for $R$ local}
We recall some structural properties of the graphs $G_R(k)$ defined in \eqref{GRk}
for $R$ a finite commutative local ring $(R,\frak m)$ with identity, where $k$ is coprime with $|R|$. 

\begin{thm}[\cite{PV9}] \label{Local case}
 Let $(R,\frak{m})$ be a finite commutative local ring with $m=|\frak{m}|$ and residue field $R/\frak{m} \simeq \ff_q$. If $k\in \mathbb{N}$ satisfies $(k,|R|)=1$, then
\begin{equation} \label{Local form G curv}
G_R(k) \simeq {\G(k,q)}^{(m)} \simeq G_{\ff_q}(k) \otimes \mathring{K}_{m} ,  
\end{equation} 
where ${\G(k,q)}^{(m)}$ denotes the balanced blow-up of order $m$ of $\G(k,q)$, whose independent sets are all the cosets of $\frak m$ in $R$, and $\mathring{K}_{m}$ is the complete graph of $m$ vertices with a loop added at every vertex.
In particular, $G_R(k)$ is $\frac{m(q-1)}{k}$-regular.
\end{thm}
\begin{rem}
Recently, Liu and Zhou \cite{LZ2} defined and studied the \textit{quadratic unitary Cayley graphs} 
$\mathcal{G}_R = Cay(R,T_{R})$,
where $T_R=Q_{R}\cup (-Q_R)$ and $Q_{R}=\{x^2: x\in R^*\}$ with $R$ a finite commutative ring with identity. 
Under the same conditions of the above theorem, they obtained the same decomposition, with a different technique, 
by using group theory. In its case they do not show that the independent sets of the blow-up are all the cosets of $\frak m$ in $R$ 
(see Theorems 2.3 and 2.5 in \cite{LZ2}). 
The same decomposition was shown by de Beaudrap \cite{Be} in the case $R=\mathbb{Z}_n$ and $k=2$.
\end{rem}



Recall that an integer $n$ is a \textit{primitive divisor} of $p^r-1$ if $n\mid p^r-1$ and $n\nmid p^{a}-1$ for all $1\le a<r$. 
For simplicity, as in our previous works \cite{PV7} we denote this fact by  
	\begin{equation} \label{prim div}
		n\dagger p^r-1.
	\end{equation} 
Also, it is well-known that 
	\begin{equation} \label{con cond}
	\G(k,q) \text{ is connected} \quad \Leftrightarrow \quad n \dagger q-1,
	\end{equation}
where $n$ is the regularity degree of $\G(k,q)$, that is 
	$$n=\tfrac{q-1}{k'} \quad \text{ where } \quad k'=(k,q-1).$$

\begin{coro}[\cite{PV9}] \label{-1RFq}
Let $(R,\frak{m})$ be a finite commutative local ring with $m=|\frak{m}|$ and residue field $R/\frak{m} \simeq \ff_q$. 
Let $k\in \mathbb{N}$ such that $(k,|R|)=1$. 
Then, 

\begin{enumerate}[$(a)$]
	\item $-1\in U_R(k)$ if and only if $-1\in U_{R/\frak m,k}$. In particular, $-1 \in U_R(k)$ if and only if $q$ is even or else if $q$ odd and $(k,q-1) \mid \frac{q-1}{2}$. \sk
	
	\item $G_R(k)$ is undirected if and only if $G_{R/\frak m}(k)$ is undirected. \sk 

\nopagebreak 	
	\item $G_{R}(k)$ is connected if and only if $\frac{q-1}{(k,q-1)} \dagger q-1$.
\end{enumerate}
\end{coro}
\begin{rem}
	Items ($a$) and ($b$) of the above corollary, imply that
	if $(R,\frak{m})$ is a finite local ring with $|R/\frak{m}|=q$, then 
	$$G_R(k) \text{ is undirected}\quad \Leftrightarrow \quad  \text{ $q$ even or else $(k,q-1)\mid \tfrac{q-1}{2}$ with $q$ odd}.$$
	Hence, this arythmetic condition will appear many times in the rest of the work.
\end{rem}

As a direct consequence of Proposition \ref{blow clique} and Theorem \ref{Local case} we obtain the following result.

\begin{thm}\label{Main teo}
	Let $(R,\frak{m})$ be a finite commutative local ring with $m=|\frak{m}|$ and residue field $R/\frak{m} \simeq \ff_q$.
	Let $k\in \mathbb{N}$ such that $(k,q-1) \mid \frac{q-1}{2}$.
 	If $(k,q)=1$, then
 		\begin{equation}
 				\mathcal{K}_{\ell}(G_R(k))=\mathcal{K}_{\ell}(\G(k,q))\cdot m^{\ell} \quad \text{for all $\ell\in \mathbb{N}$}.
 		\end{equation}
\end{thm}
\begin{proof}
	The hypothesis assure that $G_{\ff_{q}}(k)$ is an undirected graph, and by the above corollary $G_{R}(k)$ is undirected, as well.
	On the other hand, since $(k,q)=1$ then by Theorem \ref{Local case} we have that
	$$G_{R}(k)= \G(k,q)^{(m)}$$
	where $m$ is the size of $\frak{m}$. 
	Hence, the assertion it follows directly from Proposition \ref{blow clique}. 
\end{proof}
%

By taking into account the corollary \ref{coro clique},
Theorem \ref{Local case} implies the following consequence.

\begin{prop}\label{prop clique R}
	Let $(R,\frak{m})$ be a finite commutative local ring with $m=|\frak{m}|$ and residue field $R/\frak{m} \simeq \ff_q$.
	Let $k\in \mathbb{N}$ such that $(k,q-1) \mid \frac{q-1}{2}$.
	If $(k,q)=1$, then
	\begin{equation*}
		\omega(G_R(k))=\omega(\G(k,q)).
	\end{equation*}
\end{prop}

\section{ The number $\mathcal{K}_{3}(G_{R}(k))$ for $k$ small.} \label{sec 4}

In this section, we exploit some known values of $\mathcal{K}_{\ell}(\G(k,q))$ in order to obtain formulas for $R$ general local rings.
In \cite{DM}, as we mentioned in the preliminaries the authors obtained a general formula for $\mathcal{K}_{\ell}(\G(k,q))$ which is complicated to deal with, 
but for some cases ($k$ and $\ell$ small) these formulas become more tractable. 

We begin with the cases $\ell=3$ and $k=2,3,4$.
\begin{thm}\label{prop k small l 3}
 Let $(R,\frak{m})$ be a finite commutative local ring with $m=|\frak{m}|=q^{\beta}$ and residue field $R/\frak{m} \simeq \ff_q$.
 Let $k\in \N$ such that $k \mid \frac{q-1}{2}$ if $q$ is odd or else $k\mid q-1$ for $q$ even.
 Then, 
we have the following cases:
 \begin{enumerate}[($a$)]
 	\item ($k=2$)If $q\equiv 1 \pmod{4}$ then 
 		$$\mathcal{K}_{3}(G_{R}(2))= \frac{q^{3\beta+1}(q-1)(q-5)}{48}$$
	\item ($k=3$) If $q=p^r$ for a prime $p$, such that $3\mid q- 1$ if $q$ is even,
	or else $6\mid q-1$ if $q$ is odd. When $p\equiv 1 \pmod{3}$, 
	write $4q= c^2 + 27 d^2$
	for $c,d \in \Z$ such that $c\equiv 1 \pmod{3}$ and $p\nmid c$. 
	When $p\equiv 2 \pmod{3}$, let $c = -2(-p)^{\frac{r}{2}}$. Then
	 	$$\mathcal{K}_{3}(G_{R}(3))= \frac{q^{3\beta+1}(q-1)(q+c-8)}{162}.$$
	\item ($k=4$) Let $q=p^r\equiv 1\pmod{8}$ for a prime $p$. 
	Write $q=e^2+4f^2$ for $e,f\in\Z$, such that $e\equiv 1\pmod{4}$, and $p\nmid e$ when $p\equiv 1\pmod{4}$. 
	Then 		
		 	$$\mathcal{K}_{3}(G_{R}(4))= \frac{q^{3\beta+1}(q-1)(q-6e-11)}{2^7\cdot 3}.$$
 \end{enumerate}
\end{thm}
\begin{proof}
All of the assertion follows directly from Theorem \ref{Main teo} and Corollaries 2.10, 2.11 and 2.12 from \cite{DM}. 
\end{proof}
\begin{rem}
\begin{enumerate}[($a$)]
	\item In \cite{BB}, the authors found the same value in the item $(a)$, only when $R=\Z_{p^{\alpha}}$,
	by using some calculations in terms of Jacobi sums and Dirichlet characters.
	Notice that in this case its maximal ideal has size $p^{\alpha-1}$ and the residue field is $\ff_{p}$, 
	so $q=p$ and $\beta=\alpha-1$, 
	in terms of the notation of the above proposition. 
	Hence, for $p\equiv 1\pmod{4}$ we have that  
	$$\mathcal{K}_{3}(G_{\Z_{p^{\alpha}}}(2))= \frac{p^{3\alpha-2}(p-1)(p-5)}{48}.$$
	\item Notice that the hypothesis that the author assume in Corollary 2.12 of \cite{DM} was $q=e^2+f^2$
	instead of $q=e^2+4f^2$, both hypothesis are the same, since if we assume that 
	$q=e^2+f^2$ with $q\equiv 1 \pmod{8}$ and $e\equiv 1 \pmod{4}$, 
	then necessarily $f\equiv 0 \pmod{4}$, that is $f=2f'$ and so $q= e^2 +4 f'^2$. 
\end{enumerate}

\end{rem}

\subsubsection*{Derived values of $\mathcal{K}_{3}(\G(k,q))$ in field extensions for $k=3,4$.}

We can say more things about the cases $k=3,4$ when $R=\ff_q$ is a finite field. 
In \cite{PV8} Podest\'a and Videla studied the energy and spectra of $\G(3,q)$ and $\G(4,q)$,
in this case somo constants appear, very similar to the costants $c$ and $e$ present in the formulas for $\mathcal{K}_{3}(G_{\ff_{q}}(4))$.
By using complex numbers, they proved that the costants can be obtained recursively when $q$ grows (in some particular way).
This allow them to proved that the spectrum can be obtained recursively. 
In this case, we can do the same for the number of cliques, as the following results assert.

\begin{thm} \label{clique der 3 gen}
	Let $p$ be a prime with $p\equiv 1 \pmod{3}$.
	If there is a minimal $t\in \N$ such that 
	\begin{equation} \label{eq pt=x227y2}
		p^t=X^2+27 Y^2
	\end{equation}
	has integral solutions $x,y \in \Z$ with $(x,p)=1$, 
	then $\mathcal{K}_{3}(\G(3,p^{t\ell+s}))$, with $\ell\ge 1$ and $0\le s<t$, 
	is determined by the numbers $\mathcal{K}_{3}(\G(3,p^{t}))$ and $\mathcal{K}_{3}(\G(3,p^{s}))$.
\end{thm}
\begin{proof}
	Let $t$ be minimal in $\N$ such that \eqref{eq pt=x227y2} has an integral solution $x,y$ with $(x,p)=1$. Notice that if $(x,y)$ is a solution of \eqref{eq pt=x227y2} then $(x,-y)$ and $(-x, \pm y)$ are also solutions. Also, from \eqref{eq pt=x227y2} we have that $x^2\equiv 1 \pmod 3$ since $p\equiv 1 \pmod 3$ and hence $x\equiv \pm 1 \pmod 3$. Thus, we will choose one solution $(x_0,y_0)$, with $x_0 \in \{\pm x\}$ and $y_0 \in \{\pm y\}$, such that $x_0\equiv 1 \pmod 3$.
	
	Considering the complex number 
	\begin{equation} \label{zxy}
		z_{x,y}:=x + 3\sqrt{3} i y, 
	\end{equation}
	we have that $\| z_{x,y} \|^2 = x^2+27y^2 = p^t$ and hence 
	\begin{equation} \label{ptl}
		p^{t\ell} = \|z_{x,y}\|^{2 \ell} = \|z_{x,y}^{\ell}\|^2
	\end{equation}
	for any $\ell \in \N$. Now, we will express $z_{x,y}^\ell$ in the form given in \eqref{zxy}. 
	For any $\ell \in \N$ put 
	$$z_{x,y}^{\ell} := z_{x_{\ell-1},y_{\ell-1}} = x_{\ell-1}+ 3\sqrt{3}i y_{\ell-1}$$
	where $ z_{x,y}^1=z_{x,y}$ and $x_0=x$, $y_0=y$.
	For instance, $x_1+3\sqrt 3 i y_1 = z_{x,y}^2 = (x^2-27y^2)+3\sqrt 3 i (2xy)$ so $x_1=x^2-27y^2$ and $y_1=2xy$.
	By the relation $z_{x,y}^{\ell+1} = z_{x,y}z_{x,y}^\ell$, one sees that
	the sequence $\{(x_{\ell},y_{\ell})\}_{\ell \in \mathbb{N}_0}$ is thus recursively defined as follows: let $x_0=x$, $y_0=y$ and 
	for any $\ell>0$ take 
	\begin{equation} \label{recursion t1}
		x_{\ell} = x_0 x_{\ell-1} - 27 y_0 y_{\ell-1} \qquad \text{and} \qquad y_{\ell} = x_0 y_{\ell-1} + x_{\ell-1}y_0.
	\end{equation}
	Now, in the proof of Theorem 3.1 of \cite{PV8} the authors showed the following claim:
	
	\noindent
	\textit{Claim 1:} $x_{\ell} \equiv 1\pmod{3}$ and $(x_{\ell},p)=1$ for all $\ell\in \mathbb{N}_0$.
	
	\smallskip 
	
	Now assume that $s\in \{1,\ldots,t-1\}$ (the case $s=0$ was treated before), and let $c_{0,s},d_{0,s}\in \mathbb{Z}$ with $c_{0,s}\equiv 1\pmod{3}$ and $(c_{0,s},p)=1$ 
	such that 
	$$4p^s = c_{0,s}^2+27d_{0,s}^2 = \| z_{c_{0,s},d_{0,s}} \|^2$$
	with $z_{c_{0,s},d_{0,s}} = c_{0,s} + 3\sqrt{3} i d_{0,s}$. Hence, we have that
	$$4p^{t\ell+s} =\|z_{c_{0,s},d_{0,s}} \|^2 \|z_{x_{\ell-1}, y_{\ell-1}} \|^2 = \|z_{c_{0,s},d_{0,s}}  z_{x_{\ell-1},y_{\ell-1}} \|^2 = \|z_{c_{\ell,s},d_{\ell,s}} \|^2$$
	where $\{(c_{\ell,s},d_{\ell,s})\}_{\ell\in\mathbb{N}_0}$ also satisfies the recursions
	\begin{equation}\label{aies}
		c_{\ell,s}=c_{0,s} x_{\ell-1}-27 d_{0,s} y_{\ell-1} \qquad \text{and}\qquad d_{\ell,s}=c_{0,s} y_{\ell-1}+d_{0,s} x_{\ell-1},
	\end{equation}
	with $x_{\ell},y_{\ell}$ recursively defined as in \eqref{recursion t1}. The following claim is also proved in Theorem 3.1 of \cite{PV8}
	
	\noindent
	\textit{Claim 2:} $c_{\ell,s}\equiv 1\pmod{3}$ and $(c_{\ell,s},p)=1$ for all $\ell\in \mathbb{N}_0$.
	
	\smallskip
	In order to prove that the $\mathcal{K}_{3}(\G(3,p^{t\ell+s}))$ is determined by $\mathcal{K}_{3}(\G(3,p^{s}))$ and $\mathcal{K}_{3}(\G(3,p^{t}))$, 
	it is enough to put $c_{\ell,s}$ in terms of $c_{0,s}$ and $x_0$.
	In \cite{PV8} it is shown that $c_{\ell,s}$'s satisfy the following recursion 
	\begin{equation}\label{recursion imp a}
		c_{\ell+1,s}=2x_0 c_{\ell,s}-p^t c_{\ell-1,s}.
	\end{equation}
	By solving this two terms linear recurrence, we obtain that
	\begin{equation} 
		 c_{\ell,s} = \tfrac 12 (c_{0,s}+3\sqrt{3}d_{0,s} i) (x_0+3\sqrt{3}y_0 i)^{\ell} + 
			\tfrac{1}{2}(c_{0,s}-3\sqrt{3}d_{0,s} i)(x_0-3\sqrt{3}y_0 i)^{\ell}, 
	\end{equation}
	In this way, for every $\ell \in \N$, $c_{\ell,s}$ can be put in terms of $c_{0,s},d_{0,s}$ and $x_0,y_0$ only, 
	to finish the proof notice that $d_{0,s}$ can be put in terms of $c_{0,s}$ and $y_0$ can be put in terms of $x_0$, 
	as we wanted to show.
\end{proof}

Recall that an integer $a$ is a \textit{cubic residue} modulo a prime $p$ if $a\equiv x^3 \pmod p$ for some integer $x$. 
By Euler's criterion, $a$ is a cubic residue mod $p$, with $(a,p)=1$, if and only if 
\begin{equation} \label{ap cond}
	a^{\frac{p-1}{d}}\equiv 1 \pmod p
\end{equation}
where $d=(3,p-1)$. 
We have the following direct consequence of Theorem \ref{clique der 3 gen}.
\begin{thm} \label{res cubic}
	Let $p$ be a prime with $p\equiv 1 \pmod{3}$.
	If $2$ is a cubic residue modulo $p$, then the number $\mathcal{K}_{3}(\G(3,p))$ determines 
	the numbers $\mathcal{K}_{3}(\G(3,p^{\ell}))$ for every $\ell \in \N$. 
	In this case, $\mathcal{K}_{3}(\G(3,p^{\ell}))$ is given by
	\begin{equation} \label{spec p3l}
		\mathcal{K}_{3}(\G(3,p^{\ell})) = \frac{p^{\ell}(p^{\ell}-1)(p^{\ell}+c_{\ell}-8)}{162}.
	\end{equation}
	where $c_\ell$ are defined by
	\begin{equation*}
	 c_{\ell} = - (x_0+3\sqrt{3}y_0 i)^{\ell} -(x_0-3\sqrt{3}y_0 i)^{\ell}, 
	\end{equation*}
	where $x_0$ and $y_0$ are the solutions of $p=X^2+27 Y^2$ with $(x_0,p)=1$ and $x_0\equiv 1 \pmod{3}$.
\end{thm}
\begin{proof}
A classic result in number theory, conjectured by Euler and first proved by Gauss using cubic reciprocity, asserts that 
(see for instance \cite{C})
\begin{equation} \label{2res conds}
	p=x^2+27 y^2 \quad \text{for some $x,y\in\mathbb{Z}$} \qquad \Leftrightarrow \qquad 
	\begin{cases}
		p\equiv 1\pmod{3} \quad \text{ and,} \\ 2 \text{ is a cubic residue modulo $p$}.
	\end{cases}
\end{equation}
By hypothesis we have that $p\equiv 1\pmod{3}$ and $2$ is a cubic residue modulo $p$, so there exist 
$x,y \in\mathbb{Z}$ such that $p=x^2+27 y^2$. Moreover, since either $x$ or $-x$ is congruent to $1$ mod $p$, we choose the solution $(z,y)$, where $z\in \{\pm x\}$ with $z\equiv 1 \pmod 3$.
Thus, the assertion follows directly from Theorem \ref{clique der 3 gen} with $t=1$ and $s=0$.
\end{proof}
\begin{exam}  
	Let $p=31$. We know that $2$ is a cubic residue modulo $31$ and in this case we have 
	$31= 2^2 +27\cdot 1^2$. 
	We take the solutions $x_0=-2$ and $y_0=1$ of $31=X^2+37Y^2$.
	By Theorem \ref{res cubic}, we have that
	$\mathcal{K}_{3}(\G(3,31))$ determines  $\mathcal{K}_{3}(\G(3,31^{\ell}))$ for every $\ell$ and
	\begin{equation*} 
			\mathcal{K}_{3}(\G(3,p^{\ell})) = \frac{p^{\ell}(p^{\ell}-1)(p^{\ell}+c_{\ell}-8)}{162}.
	\end{equation*}
	where $c_\ell$ satisfies
	\begin{equation*} 
		c_{\ell} = - (-2+3\sqrt{3} i)^{\ell} -(-2-3\sqrt{3} i)^{\ell}, 
	\end{equation*}
	In Table 1 we give the values of $\G(3,31^{\ell})$ for the first five values of $\ell$. 
	\begin{table}[h!]
		\caption{First values of $c_{\ell}$ and $\mathcal{K}_{3}(\G(3,31^{\ell}))$}
		\begin{tabular}{|c|c|c|c|}
			\hline
			$\ell$ & $c_{\ell}$  & Values of $\mathcal{K}_{3}(\G(3,31^{\ell}))$\\ \hline 
			1 & $4$  & $155$\\
			2 & $46$ & $5689120$\\
			3 & $-308$  & $161470943875$\\ 
			4 & $-194$  & $4861047204287040$\\
			5 & $10324$  & $144899484304503423275$\\
			\hline 
		\end{tabular}
	\end{table}
%
	\hfill $\lozenge$
\end{exam}

\begin{thm} \label{clique der 4}
	If $p$ is a prime with $q=p^{r}\equiv 1 \pmod{8}$  and $p\equiv 1 \pmod{4}$,
	then the number of cliques $\mathcal{K}_{3}(\G(4,q^{\ell}))$ is determined by 
	$\mathcal{K}_{3}(\G(4,q))$ for every $\ell \in \N$. 
	Moreover, $\mathcal{K}_{3}(\G(4,q^{\ell}))$ is given by
	$$\mathcal{K}_{3}(\G(4,q^{\ell}))= \frac{q^{\ell}(q^{\ell}-1)(q^{\ell}-6e_{\ell}-11)}{2^7\cdot 3}.$$
	where the number $e_\ell$ is given by 
	\begin{equation} \label{rec cl,dl}
		e_{\ell}= \tfrac 12 (e_1+2 f_1 i)^{\ell}+ \tfrac 12 (e_1-2 f_1 i)^{\ell}=\mathit{Re} (e_1+2 f_1 i)^{\ell}, 
	\end{equation}
	where $e_1$ and $f_1$ are integral solutions of $q=X^2+ 4Y^2$ with $e_1 \equiv 1 \pmod 4$ and $(e_1,p)=1$.
\end{thm}
\begin{proof}
	It is well known that the equation $q=X^2+ 4 Y^2$ with $p\equiv 1\pmod{4}$ always has a solution $(x,y)$ satisfying $(x,p)=1$. 
	In particular, since $q\equiv 1 \pmod{8}$ then $q\equiv 1 \pmod{4}$ as well.
	Let $e_1, f_1$ be the solution of the above equation with $e_1\equiv 1\pmod{4}$ and $(e_1,p)=1$. 
	Notice that if we take 
	$z_{x,y}=x+iy$, then $p = \|z_{e_1,f_1} \|^2$, so we have that 
	$$q^{\ell} = \|z_{e_1,f_1}^{\ell}\|^2.$$
	As in the proof of Theorem \ref{clique der 3 gen}, 
	we can put 
	$z_{e_{1},f_{1}}^{\ell}=: z_{e_{\ell},f_{\ell}}$, where $e_{\ell},f_{\ell}$ are defined recursively as follows
	\begin{equation}\label{recursion t4}
		e_{\ell+1}=e_1 e_{\ell}-4f_{1} f_{\ell} \qquad \text{and} \qquad f_{\ell+1}= e_1 f_{\ell}+f_1 e_{\ell}.
	\end{equation}	
	Both sequences $\{e_{\ell}\}_{\ell \in \N_0}$ and $\{f_{\ell}\}_{\ell \in \N_0}$ also satisfy the recursion 
	\begin{equation}\label{recursion cd4}
		r_{\ell+1}=2e_1 r_{\ell}- q \cdot r_{\ell-1}.
	\end{equation}		
	It can be shown that
	$$f_{\ell+1}=f_1 \Big(\sum_{i=1}^{\ell} e_{i} e_1^{\ell-i} + e_{1}^{\ell} \Big) \quad \text{and} \quad  
	e_{\ell} = \tfrac{1}{f_1}(f_{\ell+1} - e_1 f_{\ell}) = \sum_{i=1}^{\ell} e_{i} e_1^{\ell-i} -e_1 \sum_{i=1}^{\ell-1}e_{i} e_1^{\ell-1-i} $$
	so we have that
	\begin{equation}\label{e ies 4}
		e_{\ell+1}= e_1 \sum_{i=1}^{\ell} e_{i} e_1^{\ell-i}  - q \sum_{i=1}^{\ell-1}e_{i} e_1^{\ell-1-i}- 4f_{1}^2 e_{1}^{\ell-1}.
	\end{equation} 
	As in proof of Theorem 3.1 of \cite{PV8}, we can show that $(e_{\ell},p)=1$ and $e_{\ell}\equiv 1\pmod{4}$.
	Indeed, notice that \eqref{recursion t4} implies that $e_{\ell+1}\equiv e_{\ell}\pmod{4}$
	and by hypothesis $e_1\equiv 1 \pmod{4}$ and so $e_{\ell}\equiv 1 \pmod{4}$ for all $\ell\ge 1$.
	 
	On the other hand, we have that $(e_1,p)=1$ by hypothesis.
	Notice that $e_2=e_{1}^2-4 f_{1}^2$ and $f_2=2e_{1}f_{1}$. 
	By taking into account that $q=e_{1}^2+4 f_{1}^2$, 
	we obtain that $e_2=2 e_{1}^2 -q \equiv 2e_{1}^2 \pmod{p}$. 
	Since $p>4$ is prime and $(e_1,p)=1$, we obtain that $e_2\not\equiv 0\pmod{p}$  and thus $(e_2,p)=1$.
	We now prove that $(e_\ell,p)=1$ for any $\ell \ge 3$ by contradiction.
	
	Suppose that the second statement of the claim is false, so there exists a minimum $L>2$ such that $p\mid e_L$, that is $e_{L}\equiv 0 \pmod{p}$. 
	By \eqref{e ies 4}, we obtain that 
	$$ e_1 \sum_{i=1}^{L-1} e_{i} e_1^{L-1-i} - 4f_{1}^2 e_{1}^{L-2} \equiv e_{L} \equiv 0\pmod{p},$$
	and using that $4 f_{1}^2= q-e_1^2$ we get
	$$e_1 \sum_{i=1}^{L-1} e_{i} e_1^{L-1-i} + e_{1}^{L} = e_1 \big(e_1^{L-1} +\sum_{i=1}^{L-1} e_{i} e_1^{L-1-i}  \big) 
	\equiv 0\pmod{p}.$$
	Since $(e_1,p)=1$, we have that 
	\begin{equation} \label{cong 1}
		e_1^{L-1} +\sum_{i=1}^{L-1} e_{i} e_1^{L-1-i}\equiv 0\pmod{p}.
	\end{equation}
	Notice that 
	$$\sum_{i=1}^{L-1} e_{i} e_1^{L-1-i}= e_{L-1}+\sum_{i=1}^{L-2} e_{i} e_1^{L-1-i}.$$
	By applying \eqref{e ies 4} with $\ell=L-1$ we arrive at
	$$e_{L-1}\equiv e_1 \sum_{i=1}^{L-2} e_{i} e_1^{L-2-i} - 4f_{1}^2 e_{1}^{L-3} \equiv \sum_{i=1}^{L-2} e_{i} e_1^{L-1-i} + e_{1}^{L-1} \pmod{p}$$
	where we again used that $4 f_1^2 = q-e_1^2$.
	Thus, we have that
	$$2 e_{L-1}\equiv e_{L-1} + \sum_{i=1}^{L-2} e_{i} e_1^{L-1-i} + e_{1}^{L-1} \equiv e_{1}^{L-1} + \sum_{i=1}^{L-1} e_{i} e_{1}^{L-1-i} \equiv 0\pmod{p},$$
	by \eqref{cong 1}. Hence $e_{L-1}\equiv 0 \pmod{p}$ since $(2,p)=1$, which contradicts the minimality of $L$. 
	Therefore $(e_{\ell},p)=1$ for all $\ell\in \mathbb{N}$. This proves the claim. 
	
	In order to prove that $\mathcal{K}_{3}(\G(4,q^{\ell}))$ is determined by $\mathcal{K}_{3}(\G(4,q))$, 
	notice that if $q\equiv 1 \pmod{8}$ then $q^{\ell}\equiv 1 \pmod{8}$ as well, 
	so by item ($c$) of Proposition \ref{prop k small l 3}, it is enough to put every $e_{\ell}$ in terms of $e_1$ and $f_1$ only, since the value.
	By solving the linear recurrence \eqref{recursion cd4} and by recalling that $e_2=e_{1}^2-4 f_1^2$ and $f_2=2e_1 f_1$, 
	we obtain that  $e_\ell$ are as given in \eqref{rec cl,dl}.
	Therefore, the value of $\mathcal{K}_{3}(\G(4,q^{\ell}))$ is determined by $\mathcal{K}_{3}(\G(4,q))$, as desired.	
\end{proof}

\begin{exam}
	Let $p=17$. Since $17=1^2+ 4 \cdot 2^2$, we take $e_1=1$ and $f_1=2$.
	The number of cliques $\mathcal{K}_{3}(\G(4,5^{\ell}))$ is given for any $\ell \in \N$ by 
	$$\mathcal{K}_{3}(\G(4,17^{\ell}))= \frac{17^{\ell}(17^{\ell}-1)(17^{\ell}-6e_{\ell}-11)}{2^7\cdot 3}$$
	with  
	$$e_{\ell}= \tfrac 12 (1+4 i)^{\ell}+ \tfrac 12 (1-4 i)^{\ell}=\mathit{Re} (1+4 i)^{\ell}$$ 
	In Table 2 we give the values of $\mathcal{K}_{3}(\G(4,5^{4\ell}))$ for the first five values of $\ell$ 
	\begin{table}[h!]
		\caption{Values of $\mathcal{K}_{3}(\G(4,17^{\ell}))$}
		\begin{tabular}{|c|c|c|}
			\hline
			$\ell$ & $e_{\ell}$ & Values of $\mathcal{K}_{3}(\G(4,17^{\ell}))$\\ \hline 
			1 & $1$  & $0$\\
			2 & $-15$ & $79764$\\
			3 & $-47$ & $325790856$\\ 
			4 & $161$ & $1499479239720$\\
			5 & $761$  & $7430192286281890$\\
			\hline 
		\end{tabular}
	\end{table}
%
	\hfill $\lozenge$
\end{exam}

\section{ The number $\mathcal{K}_{4}(G_{R}(k))$ for $k=2,3,4$}\label{sec 5}

In this section we study the number $\mathcal{K}_{4}(G_{R}(k))$, for $(R,\frak{m})$ a local ring for $k=2,3,4$,
by showing that we can always obtain this value in terms of $\mathcal{K}_{4}(G_{R/\frak{m}}(k))$ for $k=2,3,4$.
In the case that $R=\ff_{p^{\ell}}$, 
we show that we can alwas obtain the value of $\mathcal{K}_{4}(\G(2,p^{\ell}))$ recursively from 
$\mathcal{K}_{4}(\G(2,p))$ when $p\equiv 1 \pmod{4}$.

Given $\chi,\psi$ multiplicative characters of $\ff_{q}$ with extension $\chi(0)=\psi(0)=0$, the usual 
Jacobi sums is defined by $J(\chi,\psi)=\sum_{a\in \ff_{q}}\chi(a)\psi(1-a)$, 
we have also the symbol $\binom{\psi}{\chi}=\frac{\psi(-1)}{q}J(\chi, \overline{\psi})$.
For a multiplicative character $\chi_k$ of $\ff_{q}$ of order $k$ with $q\equiv 1\pmod{k}$ and
given $\vec{t}=(t_1,t_2,t_3,t_4.t_5)$, the hypergeometric functions 
$$\setlength\arraycolsep{1pt}
{}_3 F_2\left(\vec{t};1\right)_{q,k}=\setlength\arraycolsep{1.5pt}
{}_3 F_2\left(\begin{matrix}\chi_{k}^{t_1},& &\chi_{k}^{t_2},& &\chi_{k}^{t_3}\\[.3em] 
	&&\chi_{k}^{t_4},&&\chi_{k}^{t_5}&\end{matrix};1\right)_{q}=
\frac{q}{q-1}\sum_{\chi}\tbinom{\chi_{k}^{t_1} \chi}{\chi} \tbinom{\chi_{k}^{t_2}\chi}{\chi_{k}^{t_4}\chi}\tbinom{\chi_{k}^{t_3}\chi}{\chi_{k}^{t_5}\chi}\chi(1),$$
where the sum is taken over all of the multiplicative characters $\chi$ of $\ff_{q}$
\begin{prop}\label{prop k small l 4}
	Let $(R,\frak{m})$ be a finite commutative local ring with $m=|\frak{m}|=q^{\beta}$ and residue field $R/\frak{m} \simeq \ff_q$.
	Let $k\in \N$ such that $k \mid \frac{q-1}{2}$ if $q$ is odd or else $k\mid q-1$ for $q$ even.
	Then, we have the following cases:
	\begin{enumerate}[($a$)]
		\item $(k=2)$ 
	Let $q =p^r\equiv 1 \pmod{4}$ for a prime $p$. Write $q = x^2 +4y^2$
	for integers $x$ and $y$, such that $p\nmid x$ when $p\equiv 1 \pmod{4}$. 
	Then
		$$\mathcal{K}_{4}(G_{R}(2))=\frac{q^{4\beta+1}(q-1)((q-9)^2- 16 y^2)}{2^9\cdot 3}.$$
		\item ($k=3$) If $q=p^r$ for a prime $p$, such that $3\mid q- 1$ if $q$ is even,
		or else $6\mid q-1$ if $q$ is odd. When $p\equiv 1 \pmod{3}$, 
		write $4q= c^2 + 27 d^2$
		for $c,d \in \Z$ such that $c\equiv 1 \pmod{3}$ and $p\nmid c$. 
		When $p\equiv 2 \pmod{3}$, let $c = -2(-p)^{\frac{r}{2}}$. 
		If $\chi_3$ is a multiplicative character of $\ff_{q}$ of order $3$ and 
		$\varepsilon$ is the trivial multiplicative character,
		then
		$$\mathcal{K}_{4}(G_{R}(3))= \tfrac{q^{4\beta+1}(q-1)}{2^3 3^7}\big[ q^2+5q(c-11)+10c^2-85c+316+12 q^2 \setlength\arraycolsep{1pt}
		{}_3 F_2\left(\begin{matrix}\chi_{3},& &\chi_{3},& &\overline{\chi}_{3}\\ 
			&&\varepsilon,&&\varepsilon&\end{matrix};1\right)_{q}\big].$$
		\item ($k=4$) Let $q=p^r\equiv 1\pmod{8}$ for a prime $p$. 
		Write $q=e^2+4f^2$ for $e,f\in\Z$, such that $e\equiv 1\pmod{4}$, and $p\nmid e$ when $p\equiv 1\pmod{4}$. 
		Write $q = u^2 + 2v^2$
		for integers $u$ and $v$,
		such that $u\equiv 3 \pmod{4}$, and $p\nmid u$ when $p\equiv 1, 3 \pmod{8}$.
		If $\varphi$ and $\chi_4$ are multiplicative characters of $\ff_{q}$ of order $2$ and $4$, respectively and 
		$\varepsilon$ is the trivial multiplicative character, 
		then
		\begin{eqnarray*}
			\mathcal{K}_{4}(G_{R}(4))& = \frac{q^{4\beta+1}(q-1)}{2^{15}\cdot 3} & \cdot \Big[ q^2-2q(15x+101) + 304x^2 + (930-40u)x + 801 + 120u^2\\
			&&  + 12q^2 \setlength\arraycolsep{1pt}
			{}_3 F_2\left(\begin{matrix}\chi_{4},& &\chi_{4},& &\overline{\chi_{4}}\\ 
				&&\varepsilon,&&\varepsilon&\end{matrix};1\right)_{q} +
			30q^2 \setlength\arraycolsep{1pt}
			{}_3 F_2\left(\begin{matrix}\chi_{4},& &\varphi,& &\varphi\\ 
			&&\varepsilon,&&\varepsilon&\end{matrix};1\right)_{q} \Big]
		\end{eqnarray*} 		
	\end{enumerate}
\end{prop}
\begin{proof}
All of the assertions follows directly from Theorem \ref{Main teo} and Corollaries 2.3, 2.5 and 2.7 from \cite{DM}.
\end{proof}

\medskip
\begin{rem}
		By taking $R=\Z_{p^{\alpha}}$ in item ($a$) of Proposition \ref{prop k small l 4}, 
		we have that $\beta=\alpha-1$ and $q=p$ prime, in this case we have that
		\begin{equation}\label{valor K 4 2 Zp}
			\mathcal{K}_{4}(G_{\Z_{p^{\alpha}}}(2))=\frac{p^{4\alpha-3}(p-1)((p-9)^2- 16 y^2)}{2^9\cdot 3}.
		\end{equation}
	where $y$ is as in the hypothesis. On the other hand, in \cite{BB} the authors showed that
	$$\mathcal{K}_{4}(G_{\Z_{p^{\alpha}}}(2))=\frac{p^{2\alpha-1}(p-1)\Big(p^{2\alpha-2}\big((p-9)^2-2p\big) +J(\psi,\varphi)^{2}+\overline{J(\psi,\varphi)}^{2}\Big)}{2^9\cdot 3},$$
	where $J(\psi,\varphi)=\sum_{a\in \Z_{p^{\alpha}}}\psi(a)\varphi(1-a)$ is the Jacobi sum of $\psi$ and $\varphi$, with
	$\psi$ and $\varphi$ Dirichlet characters modulo $p^{\alpha}$ of order $4$ and 
	$2$ respectively.
	Thus, by the above equalitties we can obtain that 
		\begin{equation*}
		J(\psi,\varphi)^{2}+\overline{J(\psi,\varphi)}^{2}=2p^{2\alpha-2}(p-8y^2),	
	\end{equation*}
	This formula was found very recently in \cite{BB2} from other identity.
\end{rem}

As in theorems \ref{clique der 3 gen} and \ref{clique der 4}, we obtain the following

\begin{thm} \label{clique 4 der 2}
	If $p$ is a prime with $p\equiv 1 \pmod{4}$  
	then the number of cliques $\mathcal{K}_{4}(\G(2,p^{\ell}))$ is determined by 
	$\mathcal{K}_{4}(\G(2,p))$ for every $\ell \in \N$. 
	Moreover, $\mathcal{K}_{4}(\G(2,p^{\ell}))$ is given by
	$$\mathcal{K}_{4}(\G(2,p^{\ell}))=\frac{p^{\ell}(p^{\ell}-1)((p^{\ell}-9)^2- 16 f_{\ell}^2)}{2^9\cdot 3}.$$
	where the number $f_\ell$ is given by 
	\begin{equation} \label{rec xl,yl}
		f_{\ell}= -\tfrac{i}{4} (e_1+2 f_1 i)^{\ell}+ \tfrac{i}{4}(e_1-2 f_1 i)^{\ell}=\mathit{Im} \,(e_1+2 f_1 i)^{\ell},
	\end{equation}
	where $e_1$ and $f_1$ are integral solutions of $p=X^2+ 4Y^2$ with $e_1 \equiv 1 \pmod 4$ and $(e_1,p)=1$.
\end{thm}
\begin{proof}
It is well known that the equation $p^2=X^2+ 4Y^2$ with $p\equiv 1\pmod{4}$ always has a solution $(x,y)$ satisfying $(x,p)=1$. 
Let $e_1, f_1$ be the solution of the above equation with $e_1\equiv 1\pmod{4}$ and $(e_1,p)=1$. 
Notice that if we take 
$z_{x,y}=x+2iy$, then $p = \|z_{e_1,f_1} \|^2$, so we have that 
$$p^{\ell} = \|z_{e_1,f_1}^{\ell}\|^2.$$
As in the proof of Theorem \ref{clique der 4}, 
we can put 
$z_{e_{1},f_{1}}^{\ell}=: z_{e_{\ell},f_{\ell}}$, where $e_{\ell},f_{\ell}$ are defined recursively as follows
\begin{equation}\label{recursion 4 t2}
	e_{\ell+1}=e_1 e_{\ell}-4f_{1} f_{\ell} \qquad \text{and} \qquad f_{\ell+1}= e_1 f_{\ell}+f_1 e_{\ell}.
\end{equation}	
Both sequences $\{e_{\ell}\}_{\ell \in \N_0}$ and $\{f_{\ell}\}_{\ell \in \N_0}$ also satisfy the recursion 
\begin{equation}\label{recursion 4 cd2}
	r_{\ell+1}=2e_1 r_{\ell}- p \cdot r_{\ell-1}.
\end{equation}		
By the proof of Theorem \ref{clique der 4}, we have that $(e_{\ell},p)=1$ and $e_{\ell}\equiv 1\pmod{4}$.

In order to prove that $\mathcal{K}_{4}(\G(2,p^{\ell}))$ is determined by $\mathcal{K}_{4}(\G(2,p))$, 
notice that if $p\equiv 1 \pmod{4}$ then $p^{\ell}\equiv 1 \pmod{4}$ as well, 
so by  Proposition \ref{prop k small l 4}, it is enough to put every $f_{\ell}$ in terms of $e_1$ and $f_1$ only.
By solving the linear recurrence \eqref{recursion cd4} and by recalling that $e_2=e_{1}^2-4 f_1^2$ and $f_2=2e_1 f_1$, 
we obtain that  $f_\ell$ are as given in \eqref{rec xl,yl}.
Therefore, the value of $\mathcal{K}_{4}(\G(2,p^{\ell}))$ is determined by $\mathcal{K}_{4}(\G(2,p))$, as asserted.
\end{proof}
\begin{exam}
	Let $p=5$. Since $5=1^2+4\cdot 1^2$, we take $e_1=1$ and $f_1=1$.
	The value of $\mathcal{K}_{4}(\Gamma(2,5^{\ell}))$ is given by 
	$$\mathcal{K}_{4}(\G(2,5^{\ell}))=\frac{5^{\ell}(5^{\ell}-1)((5^{\ell}-9)^2- 16 f_{\ell}^2)}{2^9\cdot 3}.$$
	where $f_{\ell}$ is give by \eqref{rec cl,dl}, 
	$$f_{\ell}= -\tfrac{i}{4} (1+2 i)^{\ell}+ \tfrac{i}{4}(1-2i)^{\ell}=\mathit{Im}\,(1+2 i)^{\ell}.$$
	In Table 3 we give the values of $\mathcal{K}_{4}(\G(2,5^{\ell}))$ for the first five values of $\ell$. 
	\begin{table}[h!]
		\caption{Values of $\mathcal{K}_{4}(\G(2,5^{\ell}))$}
		\begin{tabular}{|c|c|c|}
			\hline
			$\ell$ & $f_{\ell}$ & Values of $\mathcal{K}_{4}(\G(2,5^{\ell}))$\\ \hline 
			1  & $1$ & $0$\\
			2 & $2$ & $75$\\
			3  & $1$ & $135625$\\ 
			4  & $22$ & $283140000$\\
			5  & $-19$ & $61674593750$\\
			\hline 
		\end{tabular}
	\end{table}
%
\end{exam}

%
%
%
%
%


\begin{thebibliography}{XXX}
	
	
	
	
	
	
	%
	
	
	
	
	\bibitem{Ak+} 
	\textsc{R.\@ Akhtar, M.\@ Boggess, T.\@  Jackson-Henderson, I.\@ Jim\'enez, R.\@ Karpman, A.\@ Kinzel, D.\@ Pritikin}.
	\textit{On the unitary Cayley graph of a finite ring}. 
	Electron.\@ J.\@ Combin.\@ \textbf{16:1}, (2009), Res.\@ Paper 117, 13 pp. 
	
	
	
	
	
\bibitem{ABLMP}
\textsc{R.\@ Atanasov, M.\@ Budden, J.\@ Lambert, K.\@ Murphy, A.\@ Penland}, 
\textit{On certain induced subgraphs of Paley graphs}. 
Acta Univ.\@ Apulensis Math.\@ Inform.\@ 40 (2014), 51--65.	
	
	
	
	
%

	\bibitem{Be}
	\textsc{N.\@ de Beaudrap}.
	\textit{On restricted unitary Cayley graphs and symplectic transformations modulo $n$}
	Electron.\@ J.\@ Combin.\@ \textbf{17}, (2010) p. R69.
	

\bibitem{BB}
\textsc{A.\@ Bhowmik, R.\@ Barman} 
\textit{On a Paley-Type Graph on $\Z_n$}
Graphs and Combinatorics \textbf{38}, 41 (2022).

\bibitem{BB2}
\textsc{A.\@ Bhowmik, R.\@ Barman} 
\textit{Cliques of orders three and four in the Paley-type graphs}
(2023) \url{arXiv:2301.07021}.
	


	
	
\bibitem{C} \textsc{D.A.\@ Cox}. \textit{Primes of the form $x^2+n y^2$}. 
John Wiley \& Sons, Inc.  1989.	
	
	
%
%
%
%



	
	

\bibitem{Da} 
\textsc{A. Das}. 
\textit{Paley-type graphs of order a product of two distinct primes}. 
Algebra and Discrete Mathematics \textbf{28} 1 (2019).

\bibitem{DM}
\textsc{M.L. Dawsey and D. McCarthy}, 
\textit{Generalized Paley graphs and their complete subgraphs of
orders three and four}, 
Res.\@ Math Sci.\@ \textbf{8}, (2021) Article number 18.
	
	
%

\bibitem{EPS} 
\textsc{R.J.\@ Evans, J.R.\@ Pulham, J.\@ Sheenan}.
\textit{On the Number of Complete Subgraphs Contained in Certain Graphs}.
 J.\@ Combin.\@  Theory Series B \textbf{30}, (1981) 364--371.
	

	
	
%
%
%
%


	
	
	
	
%
	
	
	\bibitem{Il} 
	\textsc{A.\@ Ili\'c}. 
	\textit{The energy of unitary Cayley graphs}. 
	Linear Algebra Appl.\@ 431, (2009) 1881--1889.
	
%
%

	


	\bibitem{Ki+} \textsc{D.\@ Kiani, M.M.\@ Haji Aghaei, M.\@ Yotsanan, B.\@ Suntornpoch}.
	\textit{Energy of unitary Cayley graphs and gcd-graphs}. 
	Linear Algebra Appl.\@ \textbf{435:6}, (2011) 1336--1343. 

%
%

%
%
	
	
	
	
	
	
	
	
	%
	
	\bibitem{LP} 
	\textsc{T.K.\@ Lim, C.E.\@ Praeger}. 
	\textit{On generalised Paley graphs and their automorphism groups}. 
	Michigan Math.\@ J.\@ \textbf{58:1}, (2009) 294--308.
	
	
	\bibitem{LZ} 
	\textsc{X.\@ Liu, S.\@ Zhou}. 
	\textit{Spectral properties of unitary Cayley graphs of finite commutative rings}. 
	Electron.\@ J.\@ Combin.\@ \textbf{19:4}, (2012), Paper 13, 19 pp.
	
	\bibitem{LZ2} 
	\textsc{X.\@ Liu, S.\@ Zhou}. 
	\textit{Quadratic unitary cayley graphs of finite commutative rings}. 
	Linear Algebra and its Appl.\@ \textbf{479}, (2015) 73--90.
	
	
	
	
	
%
	
	
	
	
	
	\bibitem{PP} \textsc{G.\@ Pearce, C.E.\@ Praeger.} 
	\textit{Generalised Paley graphs with a product structure}.
	Annals of Combinatorics \textbf{23}, (2019) 171--182.
	
	
%
%




	\bibitem{PV6} \textsc{R.A.\@ Podest\'a, D.E.\@ Videla}.
	\textit{The Waring's problem over finite fields through generalized Paley graphs}.
	Discrete Math.\@ \textbf{344}, (2021) 112324.


	\bibitem{PV5} \textsc{R.A.\@ Podest\'a, D.E.\@ Videla}.
	\textit{Integral equienergetic non-isospectral unitary Cayley graphs}, 
	Linear Algebra Appl.\@ \textbf{612}, (2021) 42--74.


	\bibitem{PV7} \textsc{R.A.\@ Podest\'a, D.E.\@ Videla}.
	\textit{A reduction formula for Waring numbers through generalized Paley graphs},
	Journal of Algebraic Combinatorics, \textbf{56:4}, (2022), 1255--1285. 


	\bibitem{PV8} \textsc{R.A.\@ Podest\'a, D.E.\@ Videla}.
	\textit{Generalized Paley graphs equienergetic with their complements}, 
	Linear and multilinear algebra  (2022), in press, \url{http://dx.doi.org/10.1080/03081087.2022.2159918.}

	\bibitem{PV9} \textsc{R.A.\@ Podest\'a, D.E.\@ Videla}.
	\textit{Waring numbers over finite commutative local rings}, (2022), Preprint: \url{arXiv:2212.12396}.
	
	\bibitem{PV4} \textsc{R.A.\@ Podest\'a, D.E.\@ Videla}.
	\textit{The weight distribution of irreducible cyclic codes associated with decomposable generalized Paley graphs}.
	Adv.\@ Math.\@ Comm.\@  \textbf{17:2}, (2023) 446--464. 
	
		
	
	
	
	
%
%
%
%
	
		
	\bibitem{V} \textsc{D.E.\@ Videla}.
	\textit{On diagonal equations over finite fields via walks in NEPS of graphs}.
	Finite Fields App.\@ \textbf{75}, (2021) 101882. 
	
	
%
%
%
%
	
	\bibitem{Y1} \textsc{C.H.\@ Yip}.
	\textit{On the directions determined by Cartesian products and the clique number of generalized Paley graphs}.
	Integers.\@ \textbf{21}, (2021) Paper A51 .	
	
	\bibitem{Y2} \textsc{C.H.\@ Yip}.
	\textit{On the clique number of Paley graphs of prime power order}.
	Finite Fields App.\@ \textbf{77}, (2022) 101930.
\end{thebibliography}
\end{document}